\documentclass[10pt, reqno]{amsart}
\usepackage{amsmath,amsfonts,amsthm,bm} 

\usepackage{amsfonts}
\usepackage{amssymb}
\usepackage{latexsym}

\usepackage{tikz}
\usepackage{graphicx,graphics,epsfig}
\usepackage{epstopdf}	
\usepackage[noadjust]{cite}

\DeclareMathOperator{\sech}{sech}

\usepackage[]{enumerate}
\usepackage{verbatim}
\usepackage{bbm}

\usepackage[]{enumerate}

\usepackage[sc]{mathpazo}          
\usepackage{eulervm}               
\usepackage[scaled=0.86]{berasans} 
\usepackage[scaled=1]{inconsolata} 
\usepackage[T1]{fontenc}

\usepackage[final]{hyperref}
\hypersetup{colorlinks=true, linkcolor=blue, anchorcolor=blue, citecolor=red, filecolor=blue, menucolor=blue, pagecolor=blue, urlcolor=blue}

\newcommand{\Bigabs}[1]{\Bigl\vert #1 \Bigr\vert}

\newcommand{\norm}[1]{\left\Vert #1 \right\Vert}

\newcommand{\N}{\mathbb{N}}

\newcommand{\R}{\mathbb{R}}

\newtheorem{theorem}{Theorem}

\newtheorem{lemma}{Lemma}

\newtheorem*{PWtheorem}{Paley-Wiener Theorem}

\theoremstyle{definition}

\theoremstyle{remark}

\bibliography{Bibliography}
\title[Radius of analyticity for the KDV-BBM Equation]{Improved Lower Bound for the Radius of Analyticity of Solutions to the fifth order KdV-BBM model}

  \author [
S. Mebrate] { 
Sileshi Mebrate }

\author[T. T. Dufera] {Tamirat T. Dufera}

  \author[A. Tesfahun]{Achenef Tesfahun}

\address{Department of Mathematics \\
Nazarbayev University \\
Qabanbai Batyr Avenue 53 \\
010000 Nur-Sultan \\
Republic of Kazakhstan}

\email{achenef@gmail.com}

\address{Department of Mathematics
\\
Adama Science and Technology University
\\
Ethiopia}
 
   \email{tamirat.temesgen@astu.edu.et, silenaty2005@gmail.com}

\keywords{KdV-BBM model; Global well-posedness, Improved Lower bound; Radius of analyticity; Modified Gevrey spaces}
\subjclass[2010]{35A01, 35Q53}

\begin{document}
	
		\begin{abstract}
		
 We show that the uniform radius of spatial analyticity $\sigma(t)$  of solutions at time $t$ to the fifth order KdV-BBM equation  
 cannot decay faster than $1/ \sqrt{t}$ for large $t$,  given initial data that is analytic with fixed radius $\sigma_0$. 
		 This improves a recent result by Belayneh, Tegegn and the third author \cite{BTT21}, where they obtained a $1/t$ decay of $\sigma(t)$ for large time $t$.
	\end{abstract}
	
	\maketitle

	\section{Introduction}
	
In this paper we consider the Cauchy problem for fifth order KdV-BBM  equation 
	 \begin{equation}\label{kdv-bbm}
\left\{
\begin{aligned}
	 \partial_t\eta+&\partial_x\eta-\gamma_1\partial_t \partial_x^2\eta+\gamma_2 \partial_x^3\eta+\delta_1\partial_t \partial_x^4\eta+\delta_2 \partial_x^5\eta\\&= -\frac{3}{4}\partial_x(\eta^2)-\gamma\partial_x^3(\eta^2)+\frac{7}{48}\partial_x(\eta_x^2)+\frac{1}{8}\partial_x(\eta^3),
	 \\
	 	 \eta(x,0)&=\eta_0(x),
	 \end{aligned}
\right.
	 \end{equation}
	 where 
	 $\eta:\R^{1+1}\to \R$ is the unknown function, and $\gamma,\gamma_1, \gamma_2, \delta_1, \delta_2$ are constants satisfying certain constraints; see \cite{BCS2002, CP2019} for more details.
	  The fifth order KdV-BBM  equation describes the unidirectional propagation of water waves, and was recently introduced by Bona et al. \cite{BCS2002} using the second order approximation in the two way model, the so-called  \textit{abcd}-system derived in \cite{BCS2002, BCS2004}. 
	  In the case $\gamma=\frac{7}{48}$, \eqref{kdv-bbm} satisfies the energy conservation
	 \begin{equation}\label{Energy-c}
	 E(t):=\frac{1}{2}\int_{\R}\eta^2+\gamma_1\eta_x^2+\delta_1\eta_{xx}^2dx=E(0)    \qquad (t>0).
	 \end{equation}

	The well-posedness theory for the Cauchy problem \eqref{kdv-bbm} was studied by Bona et al. in \cite{BCPS2018}, where they established local well-posedness for the initial data $\eta_0\in H^s(\R)$ with $s\geq1$.  For $\gamma_1, \sigma_1 > 0$ and $\gamma=7/48$,  the authors  \cite{BCPS2018} used the conservation of energy to prove global well-posedness of \eqref{kdv-bbm} for $\eta_0\in H^s(\R)$ with $s\geq 2$. Furthermore, they used the method of \textit{high-low frequency splitting} to obtain global well-posedness for $\eta_0\in H^s(\R)$ with $3/2\leq s<2$. The global well-posedness result was further improved in \cite{CP2019} for $\eta_0\in H^s(\R)$ with $s\geq1$.

   The main concern of this paper is to study the property of spatial analyticity of the solution $\eta(x,t)$ to \eqref{kdv-bbm}, given a real analytic initial data $\eta_0(x)$ with uniform radius of analyticity $\sigma_0$,  so that there is a holomorphic extension to a
	 complex strip \[S_{\sigma_0}=\lbrace x+iy\in \mathbb{C}:|y|<\sigma_0\rbrace.\]
	 Information about the domain of analyticity of a solution
to a PDE
can be used to gain a quantitative understanding
of the structure of the equation,
and to obtain insight into underlying physical processes. 
It is classical since the work of Kato and Masuda \cite{KM1986} that, for solutions of nonlinear dispersive PDEs with analytic initial data, the radius of analyticity, $\sigma(t)$, of the solution might decrease with $t$. 
 Bourgain \cite{B1993} used a simple argument in the context of Kadomtsev Petviashvili equation to show that $\sigma(t)$ decays exponentially in $t$. 
 
 Rapid progress has been made lately in obtaining an algebraic decay rate of the radius, i.e., $\sigma(t) \sim t^{-\alpha}$ for some $\alpha \ge 1$, to various nonlinear dispersive PDEs, see eg., \cite{BTT21, HKS2017, ST2015, ST2017, SD2015, AT2017, AT2019, AT2019-S}. The method used in these papers was first introduced by Selberg and Tesfahun \cite{ST2015} in the context of the Dirac-Klein-Gordon equations, which is based on an approximate conservation laws and Bourgan's Fourier restriction method. 
For earlier studies
concerning properties of spatial analyticity of solutions for a large class of nonlinear partial differential equations, see eg.,
\cite{BGK2005, BGK2006, B1993, Ferrari1998, FT1989, GGYTE2015, HHP2011,  HP2012, KM1986, Levermore1997, Oliver2001, Panizzi2012}.

	  By the Paley–Wiener Theorem, the radius of analyticity of a function can be related to decay properties of its Fourier transform. It is therefore natural to take initial data in Gevrey space  $G^{\sigma,s}$ defined by the norm
	  \begin{equation*}
	  \norm{f}_{G^{\sigma,s}(\R)}= \norm{\exp(\sigma|\xi|)\langle \xi \rangle^s \widehat{f}}_{L^2_\xi(\R)} \quad (\sigma\geq 0),
	  \end{equation*} 
	  where $\langle \xi \rangle=\sqrt{1+\xi^2}$. 
	 For
$\sigma= 0$, this space coincides with the Sobolev space $H^{s}(\R)$, with norm
\[
\| f \|_{H^{s}(\R)} = \| \langle \xi \rangle^{s} \hat{f} \|_{L^2_\xi (\R)},
\]
while for $\sigma>0$, any function in
  $G^{\sigma, s}(\R)$ has a radius of analyticity of at least $\sigma$ at each point $x\in \R$.
This fact is contained in the following theorem, whose proof can be found in \cite{K1976} in the case $s = 0$; the general case follows from a simple modification.  
	   \begin{PWtheorem}\label{pw}
	   	Let $\sigma > 0$ and $s \in \R $, then the following are equivalent
   	\begin{enumerate}[(a)]
   		\item $f\in G^{\sigma,s}(\R)$,
   		\item  $f$ is the restriction to $\R$ of a function $F$ which is holomorphic in the strip
   		\[S_{\sigma} = \{ x + iy \in \mathbb{C}:\ |y| < \sigma \}.\]
   	\end{enumerate}
	   	 
	   	Moreover, the function $F$ satisfies the estimates
	   	\[\sup_{|y| < \sigma} \| F (\cdot + i y) \|_{H^s(\R)} < \infty.\]	   
	   \end{PWtheorem}

 Recently,  Carvajal and Panthee \cite{CP2020} used the Gevrey space to obtain an exponential decay on the radius of spatial analyticity $\sigma(t)$ for solution $\eta(x,t)$ to \eqref{kdv-bbm}, i.e., $\sigma(t)\sim e^{-t}$ for large $t$. 
	 This was improved, more recently, to a linear decay rate, $\sigma(t)\sim 1/t$, by Belayneh, Tegegn and the third author \cite{BTT21},  using the method of almost conservation law.
	   In the present paper, we improve the decay rate further to $\sigma(t)\sim 1/ \sqrt{t}$, by using
	   a modified Gevrey space that was introduced recently in \cite{DMT2022} and the method of almost conservation law.
 
	   The modified Gevrey space, denoted $H^{\sigma,s}(\R)$, is obtained from the Gevrey space $G^{\sigma,s}(\R)$ by replacing  the exponential weight $\exp(\sigma |\xi|)$ with the hyperbolic weight $\cosh(\sigma|\xi|)$, i.e., 
	  \[ \norm{f}_{H^{\sigma,s}(\R)}= \norm{\cosh(\sigma|\xi|)\langle \xi \rangle^s \widehat{f}}_{L^2_\xi(\R)} \qquad (\sigma\geq 0).\]   
	  Observe that
\begin{equation}
\label{cosh:est}
\frac 12 \exp(\sigma |\xi|)    \le \cosh (\sigma |\xi|) \le \exp(\sigma |\xi|),
\end{equation}
and hence the $G^{\sigma, s}(\R)$ and $H^{\sigma, s}(\R)$--norms are equivalent, i.e.,
\begin{equation}\label{GH-eqiv}
\| f \|_{H^{\sigma, s} (\R)} \sim \| f \|_{G^{\sigma, s} (\R)}= \norm{\exp(\sigma|\xi|)\langle \xi \rangle^s \widehat{f}}_{L^2_\xi(\R)}.
\end{equation}
Therefore, the statement of Paley-Wiener Theorem still holds for functions in $H^{\sigma, s}(\R)$.

Observe also that    
	for $\sigma\geq 0$ 
	 the exponential weight $ \exp(\sigma |\xi|) $ satisfies the estimate
	   \begin{equation} \label{expw-est}
	  \frac{1- \exp(-\sigma |\xi|)} {|\xi|} \le \sigma  
	   \end{equation}
	whereas
	 the hyperbolic weight $\cosh(\sigma |\xi|) $ satisfies
	   \begin{equation} \label{hypw-est}
	  \frac{ 1- [\cosh(\sigma|\xi|)]^{-1} } {|\xi|^2}\leq \sigma^2 .
	   \end{equation}
	Consequently, the decay rate $\sigma(t) \sim 1/t  $ that was obtained in \cite{BTT21} stems from the $\sigma$-factor on the r.h.s of \eqref{expw-est}  whereas the improved decay rate $\sigma(t) \sim  1/ \sqrt{t}$ obtained in this paper stems from the $\sigma^2$-factor on the r.h.s of \eqref{hypw-est}.

	  \vspace{3mm}
           
      We state our main result  as follows.
      \begin{theorem}[Asymptotic lower bound for $\sigma$]\label{thm-main}
      	Let $\gamma_1,  \delta_1>0$, $\gamma=\frac{7}{48}$ and $\eta_0\in H^{\sigma_0,2}(\R)$ for $\sigma_0 > 0$ . Then the global \footnote{ As a consequence of the embedding $H^{\sigma_0, 2}(\R)  \hookrightarrow H^{2}(\R)$ and the existing well-posedness theory in $H^2(\R)$ (see \cite{BCPS2018}), the Cauchy problem \eqref{kdv-bbm} (with $\gamma_1, \delta_1 > 0$ and $\gamma= 7/48$) has a unique, smooth solution for all time, given initial data 
$\eta_0 \in H^{\sigma_0, 2}$. } solution $\eta$ of \eqref{kdv-bbm} satisfies \[\eta(t)\in H^{\sigma, 2}(\R)\quad \text{for all} \quad t>0,\] with the radius of analyticity $\sigma$ satisfying the asymptotic  lower bound  \[\sigma:=\sigma(t)\geq C/\sqrt{t} \quad \text{as} \quad t\to \infty,\] where $C>0$ is constant depending on the initial data norm $\norm{\eta_0}_{H^{\sigma_0,2}(\R)}$.
      \end{theorem}
      So it follows from Theorem \ref{thm-main}  that the solution $\eta(t)$ at any time $t$ is analytic in the strip $S_{\sigma(t)}$ (due to \eqref{GH-eqiv} and the Paley-Wiener Theorem).
      
     \vspace{2mm}
     To prove Theorem \ref{thm-main}  first we  establish the following local well-posedness result,  which states that for short time the radius of analyticity of solution remains constant.
      \begin{theorem} {(Local well-posedness).}\label{thm-lwp}
      	Let $\sigma_0>0$ and $\eta_0\in H^{\sigma_0,2}(\R)$. Then there exist a unique solution \[\eta\in C([0,T];H^{\sigma_0,2}(\R))\] of the Cauchy problem \eqref{kdv-bbm}, where the existence time is
      	 \begin{equation}\label{T}
      	 T\sim \left(1+\norm{\eta_0}_{H^{\sigma_0,2}(\R)}  \right)^{-2}.
      	 \end{equation}
      	 Moreover, the data to solution map $\eta_0 \mapsto \eta$ is continuous from $H^{\sigma_0,2}(\R)$ to $C([0,T];H^{\sigma_0,2}(\R))$.
      	 
      \end{theorem}
  
  Next, we derive an approximate energy conservation law for \[v_{\sigma}:=\cosh (\sigma|D|)\eta,\] where  $D=-i\partial_x$ and $\eta$ is a solution to \eqref{kdv-bbm}. To do this, we define 
  a modified energy associated with $v_\sigma$ by
  \begin{equation}\label{MEnergy-c}
  E_\sigma(t)=\frac{1}{2}\int_{\R}v_{\sigma}^2+\gamma_1 (\partial_xv_\sigma)^2+\delta_1(\partial_x^2v_\sigma)^2 dx.
  \end{equation} 
   Note that since  $v_0=\eta$, by \eqref{Energy-c} we have $ E_0(t)=   E_0(0) $ for all  $t$.
   
  \begin{theorem}\label{thm-almostconv}
  	{(Almost conservation law).} Let $\eta_0\in H^{\sigma,2}(\R)$. Suppose that $\eta\in C([0,T];H^{\sigma,2}(\R))$ is the local-in-time solution to the Cauchy problem \eqref{kdv-bbm} from Theorem \ref{thm-lwp}. Then
  		\begin{equation}\label{almostconv-est}
  	\sup_{0\leq t \leq T}E_{\sigma}(t)=E_{\sigma}(0)+\sigma^2T. \mathcal{O} \left( \left[1+(E_{\sigma}(0))^{\frac{1}{2}} \right](E_{\sigma}(0))^{\frac{3}{2}}\right).
  	\end{equation} 
  \end{theorem}
  Observe that from \eqref{almostconv-est}, in the limit as $\sigma \rightarrow 0$, we recover the conservation 
  $ E_0(t)=   E_0(0) $ for $0\le t \le T$.
Applying the last two theorems repeatedly, and then by taking $\sigma$ small enough we can cover any time interval $[0, T_\ast]$ and obtain the lower bound in Theorem \ref{thm-main}.

  \vspace{2mm}
 \textbf{Notation}: For any positive numbers $p$ and $q$, the notation $p \lesssim q$ stands for
$p\leq cq$, where $c$ is a positive constant that may vary from line to line. Moreover,
we denote $p \sim q$ when $p \lesssim q$ and $q \lesssim p$.

\vspace{2mm}

In the next sections we prove Theorems \ref{thm-lwp}, \ref{thm-almostconv} and \ref{thm-main}.

\section{Proof of Theorem \ref{thm-lwp}}
 Taking the spatial Fourier transform of the first equation in \eqref{kdv-bbm} we obtain
      \begin{equation*}
      \begin{split}
      \partial_t\widehat{\eta}+i\xi\widehat{\eta}&+\gamma_1\xi^2\partial_t\widehat{\eta}-i\gamma_2\xi^3\widehat{\eta}+\delta_1\xi^4\partial_t\widehat{\eta}+i\delta_2\xi^5\widehat{\eta}\\&=-\frac{3}{4}i\xi\widehat{\eta^2}+i\gamma\xi^3\widehat{\eta^2}+\frac{7}{48}i\xi\widehat{\eta_x^2}+\frac{1}{8}i\xi\widehat{\eta^3}.
      \end{split}
      \end{equation*}
      Arranging the terms we have
      \begin{equation*}
      \begin{split}
      \left(1+\gamma_1\xi^2+ \delta_1\xi^4\right)\partial_t\widehat{\eta}&+i\xi\left(1-\gamma_2\xi^2+\delta_2\xi^4\right)\widehat{\eta}\\&=\frac{1}{4}i\xi\left(-3+4\gamma\xi^2\right)\widehat{\eta^2}+\frac{7}{48}i\xi\widehat{\eta_x^2}+\frac{1}{8}i\xi\widehat{\eta^3}.
      \end{split}
      \end{equation*} 
      Dividing this equation by $\varphi(\xi):=1+\gamma_1\xi^2+\delta_1\xi^4$ and multiplying by $i$,
      we obtain
     \begin{equation}\label{bbm-reform}
     i\partial_t \widehat{\eta}-\phi(\xi)\widehat{\eta}=\tau(\xi)\widehat{\eta^2}-\frac{7}{48}\psi(\xi)\widehat{\eta_x^2}-\frac{1}{8}\psi(\xi)\widehat{\eta^3},
     \end{equation} 
     where \[ \phi(\xi)=\frac{\xi(1-\gamma_2\xi^2+\delta_2\xi^4)}{\varphi(\xi)}, \quad \tau(\xi)=\frac{\xi(3-4\gamma\xi^2)}{4\varphi(\xi)}, \quad\psi(\xi)=\frac{\xi}{\varphi(\xi)}.\] 
      In an operator form \eqref{bbm-reform} can be rewritten as
     \begin{equation}\label{bbm-reformo}
     i\partial_t\eta-\phi(D)\eta=\underbrace{\tau(D)\eta^2-\frac{7}{48}\psi(D)\eta_x^2-\frac{1}{8}\psi(D)\eta^3}_{:= F(\eta)},
     \end{equation} 
     where
    $\phi(D)$, $\psi(D)$ and $\tau(D)$ are Fourier multiplier operators defined as \[\mathcal{F}[\phi(D)f](\xi)=\phi(\xi)\widehat{f}(\xi), \quad \mathcal{F}[\psi(D)f](\xi)=\psi(\xi)\widehat{f}(\xi), \quad \mathcal{F}[\tau(D)f](\xi)=\tau(\xi)\widehat{f}(\xi).\]
Now given initial data $\eta(0)=\eta_0$, the integral representation of \eqref{bbm-reformo} is
     \begin{equation} \label{Integeq}
     \eta(t)=e^{-it\phi(D)}\eta_0-i\int_{0}^{t}  e^{-i(t-s)\phi(D)} F(\eta)(s)ds.
     \end{equation}

By combining the estimates in \cite[Lemma 2.2--2.4]{CP2020} and \eqref{GH-eqiv}, we obtain the following a priori estimate for the $H^{\sigma,2}(\R)$-norm of $F(\eta)$.
     \begin{lemma} \label{lm-Fest}
     	For $\sigma \ge 0$, we have nonlinear estimate 
     	$$
     	\norm{F(\eta)}_{H^{\sigma,2}(\R)}\lesssim [1+\norm{\eta}_{H^{\sigma,2}(\R)}]\norm{\eta}^2_{H^{\sigma,2}(\R)} 
$$   x
for all $\eta\in H^{\sigma,2}(\R)$.

  \end{lemma}

	     Next, we use the contraction mapping techniques and Lemma \ref{lm-Fest} to prove Theorem \ref{thm-lwp}. To this end, define the mapping $\eta \mapsto \Gamma(\eta)$ by
	     	\[\Gamma (\eta)(t):=e^{-it\phi(D)}\eta_0-i\int_{0}^{t}  e^{-i(t-s)\phi(D)} F(\eta)(s)ds\] 
	and the space $X_T$ by
	     	\[X_T=C([0,T]:H^{\sigma,2}(\R))  \quad \text{ with norm} \quad \norm{u}_{X_{T}}=\sup_{0\leq t \leq T}\norm{u(t)}_{H^{\sigma,2}(\R)} . \] 
	Then we look for a solution in the set  \[\mathcal S_r=\lbrace \eta\in X_T \ : \ \norm{\eta}_{X_T}\leq  r\rbrace ,\] 
	where $2r=\norm{\eta_0}_{H^{\sigma,2}(\R)}.$
	
	For $\eta \in X_T$, we have by Lemma \ref{lm-Fest},
	\begin{equation}\label{Gamma-Est}
	\begin{split}
		\norm{\Gamma(\eta) }_{X_T}  &\leq \norm{\eta_0}_{H^{\sigma, 2}(\R)}+c T\left[1+\norm{\eta}_{X_T}\right] \norm{\eta}^2_{X_T} \\& \leq r/2 +cT r (1+r)^2.
		\end{split}
	\end{equation}
	Similarly, for $\eta_1, \eta_2 \in X_T$, we obtain the difference estimate
	\begin{equation}\label{Gamma-Diff}
	\norm{\Gamma(\eta_1) - \Gamma(\eta_2) }_{X_T}  \leq cT  (1+r)^2 \norm{\eta_1-\eta_2}_{X_T}.
	\end{equation}
By choosing
$$T = \frac{1}{2c(1+r)^2}$$
	in \eqref{Gamma-Est} and \eqref{Gamma-Diff} we obtain 
	\begin{equation*}\label{Gamma}
		\norm{\Gamma(\eta) }_{X_T}  \leq  r  \quad \text{and} \quad \norm{\Gamma(\eta_1) - \Gamma(\eta_2) }_{X_T}  \leq \frac12  \norm{\eta_1-\eta_2}_{X_T}.
	\end{equation*}
Therefore,  $\Gamma$ is a contraction on $\mathcal S_r$ and therefore it has a unique fixed point
 $\eta \in \mathcal S_r$ solving the integral equation
 \eqref{Integeq} on $\R \times [0, T]$. Continuous dependence on the initial data can 
be shown in a similar way, using the difference estimate. This concludes the proof of Theorem \ref{thm-lwp}.

	  \section{Proof of Theorem \ref{thm-almostconv}}
	 
	  Fix $\gamma_1, \delta_2>0$ and $\gamma=\frac{7}{48}$. Recall that
	  $v_{\sigma}:=\cosh(\sigma|D|)\eta$, where $\eta$ is the solution to \eqref{kdv-bbm}, and hence $\eta =\sech (\sigma |D|)  v_\sigma$.
	  
Applying the operator $\cosh(\sigma|D|)$ to \eqref{kdv-bbm} we obtain
	 \begin{equation}\label{bbm-mod}
	 \begin{split}
	 \partial_tv_{\sigma}+&\partial_xv_\sigma-\gamma_1\partial_t \partial_x^2v_{\sigma}+\gamma_2 \partial_x^3v_{\sigma}+\delta_1\partial_t \partial_x^4v_{\sigma}+\delta_2 \partial_x^5v_{\sigma} \\& =-\left(\frac{3}{4}+\gamma\partial_x^2\right) \partial_x (v_{\sigma}^2)+\gamma\partial_x({\partial_xv_\sigma})^2+\frac{1}{8}\partial_x(v_{\sigma}^3)+N(v_{\sigma}),
	 \end{split}
	 \end{equation}
	  where
	  \begin{equation}\label{N}
	  N(v_{\sigma})=\left(\frac{3}{4}+\gamma\partial_x^2\right)\partial_x N_1(v_{\sigma})-\gamma \partial_x N_2(v_{\sigma})-\frac{1}{8}\partial_x N_3(v_{\sigma})
	  \end{equation}
	  with 
	  \begin{equation}\label{Nj}
	  \begin{split}
	  N_1(v_{\sigma})&=v_{\sigma}^2-\cosh(\sigma|D|) \left[ \sech(\sigma|D|)v_\sigma \right]^2,\\ N_2(v_{\sigma})&=(\partial_xv_\sigma)^2-\cosh(\sigma|D|) \left[ \sech(\sigma|D|) \partial_ xv_\sigma \right]^2,\\N_3(v_{\sigma})&=v_{\sigma}^3-\cosh(\sigma|D|) \left[ \sech(\sigma|D|) v_\sigma \right]^3.
	  \end{split}
	  \end{equation}

Differentiating the modified energy, \eqref{MEnergy-c}, and using \eqref{bbm-mod}--\eqref{Nj} we obtain
	  \begin{align*}
	  \frac{d}{dt}E_{\sigma}(t)&=\int_{\R}v_\sigma \partial_tv_\sigma+\gamma_1 \partial_xv_\sigma \partial_t(\partial_xv_\sigma)+\delta_1 \partial_{xx}v_\sigma \partial_t(\partial_{xx}v_\sigma)dx\\&=\int_{\R}v_\sigma[\partial_tv_\sigma-\gamma_1\partial_t \partial^2_xv_\sigma+\delta_1\partial_t\partial^4_xv_\sigma ]dx\\&=-\int_{\R}v_\sigma\left[\partial_xv_\sigma+ \gamma_2\partial^3_xv_\sigma+\delta_2\partial^5_xv_\sigma+\left(\frac{3}{4}+\gamma\partial_x^2\right) \partial_x (v_{\sigma}^2)-\gamma \partial_x(\partial_xv_\sigma)^2-\frac{1}{8}\partial_x(v_\sigma)^3\right]dx\\&+\int_{\R}v_\sigma N(v_\sigma)dx.
	  \end{align*}
However,  the integral on the third line is zero  due to integration by parts (assuming sufficiently regular solution) and the following identities:
	  \begin{align*}
    u \partial_x u &= \frac12 \left(u^2\right)_x, \qquad  u \partial^3_x u = \left(u u_{xx}\right)_x  -\frac12 \left(u_x^2\right)_x,
    \\
  u \partial^5_x u &= \left(u \partial^4_{x} u\right)_x - \left(\partial_x u \partial^3_{x} u\right)_x  + \frac12 \left(u_{xx}^2\right)_x
\end{align*}
and 
\begin{align*}
   u \partial_x \left(u^2 \right)&=  \frac23 \left(u^3 \right)_x, \qquad u \partial_x \left(u^3 \right)= \frac34 \left(u^4 \right)_x  ,\\
  u \partial^3_x (u^2)&=2 \left(u^2 u_{xx} \right)_x + u \left(u_x^2 \right)_x.
\end{align*}
	    Therefore, 
	 \[\frac{d}{dt}E_{\sigma}(t)=\int_{\R}v_\sigma(x,t)N(v_\sigma(x,t))dx.\] 
	 Consequently, integrating with respect to time we get
	 \begin{equation}\label{m-energy}
	 E_{\sigma}(t)=E_{\sigma}(0)+\int_{0}^{t}\int_{\R}v_\sigma(x,s)N(v_\sigma(x,s))dxds.
	 \end{equation}

	Combining \eqref{m-energy} with the following  
 key lemma, which will be be proved in the last section, we obtain \eqref{almostconv-est}.
\begin{lemma}\label{lm-keylm}
	For  $N(v_\sigma)$ as in \eqref{N}--\eqref{Nj} we have 
	\begin{equation}
	\label{KeyEst}
	\Big| \int_{\R}v_\sigma N(v_\sigma)dx\Big| \leq c\sigma^2\left[1+\norm{v_\sigma}_{H^2(\R)}\right]\norm{v_\sigma}^3_{H^2(\R)}
		\end{equation}
	for all $v_\sigma\in H^2(\R)$.
\end{lemma}
Indeed, applying
 \eqref{KeyEst} to \eqref{m-energy} we obtain
\begin{equation}\label{apriorienergy}
\sup_{0\leq t \leq T}E_{\sigma}(t)=E_{\sigma}(0)+\sigma^2 T. \mathcal{O} \left([1+\norm{v_\sigma}_{L^{\infty}_{T}H^2}]\norm{v_\sigma}^3_{L^{\infty}_{T}H^2}\right)
\end{equation}
where $L^{\infty}_{T}H^2:=L^{\infty}_tH^2([0,T]\times \R)$ with $T$ is as in Theorem \ref{thm-lwp}.

As a consequence of Theorem \ref{thm-lwp} we have the bound 
 \begin{equation}\label{solbd}
	\norm{v_\sigma}_{L^{\infty}_{T}H^2(\R)}=\norm{\eta}_{L^{\infty}_{T}H^{\sigma,2}(\R)}\leq c\norm{\eta_0}_{H^{\sigma,2}(\R)}=c\norm{v_\sigma(\cdot, 0)}_{H^2(\R)}.
	\end{equation} 
		On the other hand, 
 \begin{equation}\label{m-engy0}
 \begin{split}
 E_{\sigma}(0)&=\frac{1}{2}\int_{\R}\left[v_\sigma(x,0)\right]^2+\gamma_1 \left[\partial_xv_\sigma(x,0)\right]^2+\delta_1\left[\partial_x^2v_\sigma(x,0)\right]^2 dx\\& \sim \norm{v_\sigma(.,0)}^2_{H^2(\R)}.
 \end{split}
 \end{equation}
From \eqref{solbd} and \eqref{m-engy0} we get
 $$\norm{v_\sigma}_{L^{\infty}_{T}H^2(\R)}\sim  \left(E_{\sigma}(0)\right)^\frac12,$$ which can combined
 with \eqref{apriorienergy} to obtain the desired estimate \eqref{almostconv-est}.

\section{Proof of Theorem \ref{thm-main}}
Suppose that $\eta(\cdot,0)=\eta_0\in H^{\sigma_0,2}(\R)$ for some $\sigma_0>0$.
This implies 
$
v_{\sigma_0} (\cdot, 0)=\cosh (\sigma_0 |D)|) \eta_0 \in H^2,
$ and hence 
\[E_{\sigma_0}(0)\sim \norm{v_{\sigma_0}(.,0)}^2_{H^2(\R)} <\infty.\]

Now following the argument in \cite{ST2015} (see also \cite{SD2015})
we can construct a solution on $[0, T_\ast]$ for arbitrarily large time $T_\ast$. This is achieved by applying the approximate conservation \eqref{almostconv-est}, so as to repeat the local result in Theorem \ref{thm-almostconv} on successive short time intervals of size $T $ to reach $T_\ast$, by adjusting the strip width parameter $ \sigma \in (0, \sigma_0]$ of the solution according to the size of $T_\ast$. 

In what follows we prove that
\begin{equation}\label{Ebound}
\sup_{0\leq t \leq T_\ast }E_{\sigma}(t)\leq 2 E_{\sigma_0}(0) \quad \text{for} \quad \sigma\ge C /\sqrt{T_\ast}
\end{equation} 
for arbitrarily large $T_\ast$ and  $C>0$ depending on $E_{\sigma_0}(0)$.
This would in turn imply 
\[ \sup_{0\leq t \leq T_\ast } \norm{\eta(t)}_{H^{\sigma, 2}(\R)} < \infty  \quad \text{for} \quad \sigma \ge C/ \sqrt{T_\ast}\] 
which proves Theorem \ref{thm-main}.

It remains to prove \eqref{Ebound}. To do this, first observe that for $\sigma \in (0, \sigma_0]$
and $\tau \in (0, T]$, we have by
Theorems \ref{thm-lwp} and 
\ref{thm-almostconv},
\begin{align*}
\sup_{0\leq t \leq \tau }E_{\sigma}(t)&\leq E_{\sigma}(0)+c\sigma^2T \left[1+(E_{\sigma}(0))^{1/2} \right](E_{\sigma}(0))^{3/2}
\\
& \leq E_{\sigma_0}(0)+c\sigma^2T \left[1+(E_{\sigma_0}(0))^{1/2} \right](E_{\sigma_0}(0))^{3/2}.
\end{align*} 
To get the second line we used the fact the $ \mathcal E_{\sigma}(0) \le \mathcal E_{\sigma_0}(0)$ which holds for $\sigma \le \sigma_0$ as $\cosh r$ is increasing for $r\ge 0$.
Thus, 
\begin{equation}\label{Ebound1}
 \sup_{0\leq t \leq \tau}E_{\sigma}(t)\leq 2E_{\sigma_0}(0)
 \end{equation}
provided that 
\begin{equation}\label{sigma-cond1}
c\sigma^2T \left[1+(E_{\sigma_0}(0))^{1/2} \right](E_{\sigma_0}(0))^{3/2} \leq E_{\sigma_0}(0).
\end{equation}

Next, we apply Theorem \ref{thm-lwp} with initial time $t=\tau$ and time-step size $T$ as in \eqref{T} to extend the solution from $[0, \tau]$ to $[\tau,\tau+T]$. By Theorem \ref{thm-almostconv} and \eqref{Ebound1} we obtain
\begin{equation}
\sup_{\tau\leq t \leq \tau+ T}E_{\sigma}(t)\leq E_{\sigma}(\tau)+c\sigma^2T\left[ 1+(2E_{\sigma_0}(0))^{1/2} \right]\left[(2E_{\sigma_0}(0))^{3/2}\right].
\end{equation}
In this way we cover all time intervals $[0,T], [T,2T]$, etc., and obtain
\begin{align*}
E_{\sigma}(T)&\leq E_{\sigma}(0)+c\sigma^2T \left[1+(2E_{\sigma_0}(0))^{1/2} \right](2E_{\sigma_0}(0))^{3/2}
\\ 
E_{\sigma}(2T)&\leq E_{\sigma}(T)+c\sigma^2T[1+(2E_{\sigma_0}(0))^{1/2}](2E_{\sigma_0}(0))^{3/2}
\\
& \leq 
E_{\sigma}(0)+ 2c\sigma^2T[1+(2E_{\sigma_0}(0))^{1/2}](2E_{\sigma_0}(0))^{3/2}
\\
& \vdots \\ E_{\sigma}(nT)&\leq E_{\sigma}(0)+nc\sigma^2T[1+(2E_{\sigma_0}(0))^{1/2}](2E_{\sigma_0}(0))^{3/2}.
\end{align*}
This argument can be continued as long as
\begin{equation}
nc\sigma^2T \left[1+(2E_{\sigma_0}(0))^{1/2} \right](2E_{\sigma_0}(0))^{3/2}\leq E_{\sigma_0}(0)
\end{equation}
as this would imply $E_{\sigma}(n T)\leq 2E_{\sigma_0}(0)$.

 Thus, the induction stops at the first integer $n$ for which 
 \[nc\sigma^2T \left[1+(2E_{\sigma_0}(0))^{1/2} \right](2E_{\sigma_0}(0))^{3/2}> E_{\sigma_0}(0)\] and then we have reached the finite time $T_\ast=nT$ when
  \[c\sigma^2T_\ast \left[1+(2E_{\sigma_0}(0))^{1/2} \right](2E_{\sigma_0}(0))^{1/2}> 1.\]
 This proves $\sigma \geq  C/ \sqrt{T_\ast}$ for some $C>0$ depending on $E_{\sigma_0}(0)$.

\section{Proof of Lemma \ref{lm-keylm} }

To prove \eqref{KeyEst} we need the following estimate from \cite[Lemma 3]{DMT2022}
in the special cases of $p=2$ and $p=3$. 
\begin{lemma} 
\label{lm-coshest}
Let $\xi=\sum_{j=1}^p \xi_j$ for $\xi_j \in \R$, where $p \ge 1$ is an integer. 
Then 
\begin{equation}\label{coshpest}
\left|1 -  \cosh |\xi| \prod_{j=1}^p  \sech |\xi_j|   \right| \le 2^p \sum_{  j\neq k =1}^p   |\xi_j| |\xi_k|.
\end{equation}
\end{lemma}

\begin{proof}
For the readers convenience we include the proof in the case $p=2$.
Note that
\begin{equation}\label{cosh-prod}
  \cosh |\xi_1| \cosh |\xi_2| = \frac12  \left[  \cosh(|\xi_1|-|\xi_2|) +  \cosh(|\xi_1|+ |\xi_2|) \right].
\end{equation}
On the other hand,  we have (see  \cite[Lemma 2]{DMT2022}),
\begin{equation}\label{coshpest0}
\left| \cosh b - \cosh a   \right| \le \frac12 \Bigabs{ b^2-a^2} \left(\cosh b+  \cosh a\right).
\end{equation}
for $a, b \in \R$.

\vspace{2mm}

Then by \eqref{cosh-prod} and \eqref{coshpest0},
\begin{equation*}\label{cosh-prodest1}
\begin{split}
\left|   \cosh |\xi_1| \cosh |\xi_2|-\cosh |\xi|  \right| 
 & =    \left|  \frac12 \left(\sum_{\pm} \cosh \left( |\xi_1|\pm  |\xi_2|\right)-\cosh |\xi|  \right)\right|
 \\
 &\le \frac12 \sum_{\pm} \frac12 \left|  \left( |\xi_1|\pm  |\xi_2|\right)^2-  |\xi|^2 \right| \left( \cosh \left( |\xi_1| \pm  |\xi_2|\right) +  \cosh |\xi|  \right)
 \\
 &\le \frac12 \cdot 4 |\xi_1| |\xi_2|  \cdot   4 \cosh(|\xi_1|) \cosh(|\xi_2|) 
\\
 &=   8  |\xi_1| |\xi_2| \cosh(|\xi_1|) \cosh(|\xi_2|) .
\end{split}
\end{equation*}
Dividing by $ \cosh(|\xi_1|) \cosh(|\xi_2|) $
 yields the desired estimate \eqref{coshpest} in the  case $p=2$.

\end{proof}

Next we prove \eqref{KeyEst}. For $N(v_\sigma)$ as in \eqref{N}--\eqref{Nj}, we use  Plancherel theorem to write
	 	\begin{align*}
	 	\int_{\R}&v_\sigma N(v_\sigma)dx=\int_{\R}v_\sigma \left(\frac{3}{4}+\gamma\partial_x^2\right)\cdot \partial_xN_1(v_\sigma)-\gamma v\partial_xN_2(v_\sigma)-\frac{1}{8}v\partial_x N_3(v_\sigma)dx\\&= \underbrace{\int_{\R}\left(\frac{3}{4}+\gamma\partial_x^2\right)v_\sigma\cdot \partial_xN_1(v_\sigma)dx}+\underbrace{\gamma\int_{\R}\partial_xv_\sigma \cdot N_2(v_\sigma)dx}+ \underbrace{\frac{1}{8}\int_{\R}\partial_xv_\sigma \cdot N_3(v_\sigma) dx}.\\&\quad\quad\quad\quad\qquad :=I_1 \qquad\qquad\qquad\qquad\qquad:= I_2 \qquad\qquad \qquad \qquad :=I_3
	 	\end{align*} 
So \eqref{KeyEst} follows from
\begin{align}
\label{Est-Ij}
|I_j| &\lesssim \sigma^2\norm{v_\sigma}^3_{H^2(\R)}, \qquad (j=1,2)
\\
\label{Est-I3}
|I_3|&\lesssim \sigma^2\norm{v_\sigma}^4_{H^2(\R)}.
\end{align}
\subsection{ Proof of \eqref{Est-Ij} when $j=1$}
By Cauchy-Schwarz inequality,
	 	\begin{align*}
	 	|I_1|&\leq \norm{ \left(\frac{3}{4}+\gamma\partial_x^2 \right)v_\sigma}_{L^2_x(\R)}\norm{\partial_x N_1(v_\sigma)}_{L^2_x(\R)}\\& \lesssim \norm{v_\sigma}_{H^2(\R)}\norm{\partial_x N_1(v_\sigma)}_{L^2_x(\R)}.
	 	\end{align*}
	 	So the proof reduces to
	 	 \begin{equation}
	 	 \label{Est-I1R}
	 	 \norm{\partial_x N_1(v_\sigma)}_{L^2_x(\R)}\lesssim \sigma^2 \norm{v_\sigma}^2_{H^2(\R)},
	 	 	 	 \end{equation}
	 	 	 	 where
	 	 	 	 $$
	 	 	 	 N_1(v_{\sigma})=v_{\sigma}^2-\cosh(\sigma|D|) \left[ \sech(\sigma|D|)v_\sigma \right]^2.
	 	 	 	 $$
Now taking the Fourier Transform of $\partial_xN_1(v_\sigma)$ and applying \eqref{coshpest} with $p=2$, we obtain
	 	\begin{align*}
	 	\Big|\mathcal F &\left[\partial_xN_1(v_\sigma)\right] (\xi)\Big|
	 	\\
	 	&
	 	=\Bigg|\int_{\xi=\xi_1+\xi_2} i\xi 
	 	\left( 1-\cosh(\sigma|\xi|)\prod_{j=1}^{2}\sech(\sigma|\xi_j|) \right)  \widehat{v_\sigma}(\xi_1)\widehat{v_\sigma}(\xi_2) d\xi_1d\xi_2 \Bigg| \\
	 	&\leq 4\sigma^2 \int_{\xi=\xi_1+\xi_2} |\xi| \left( \sum_{  j\neq k =1}^2  |\xi_j| |\xi_k| \right)| |\widehat{v_\sigma}(\xi_1)||\widehat{v_\sigma}(\xi_2)|d\xi_1d\xi_2.
	 	\end{align*} 
	 	
By symmetry, we may assume	$|\xi_1|\leq |\xi_2|$. Then
	 	\begin{align*}
	 	\Big|\mathcal F \left[\partial_xN_1(v_\sigma)\right] (\xi)\Big|&\leq 16\sigma^2\int_{\xi=\xi_1+\xi_2}|\xi_1||\widehat{v_\sigma}(\xi_1)| \cdot |\xi_2|^2  |\widehat{v_\sigma}(\xi_2)| d\xi_1d\xi_2\\&=16\sigma^2\mathcal{F}_x[|D|w_\sigma.|D|^2w_\sigma](\xi),
	 	\end{align*}
	 	where $w_\sigma=\mathcal{F}^{-1}_x(|\widehat{v_\sigma}|)$.
	 	Finally, by  Plancherel, H\"older inequality and  Sobolev embedding,
	 	\begin{align*}
	 	\norm{\partial_xN_1(v_\sigma)}_{L^2_x(\R)}&\le 16\sigma^2\norm{ |D|w_\sigma.|D|^2 w_\sigma  }_{L^2_x(\R)}
	 	\\&\lesssim\sigma^2 \norm{|D|w_\sigma}_{L^{\infty}_x(\R)} \norm{|D|^2w_\sigma}_{L^2_x(\R)}\\& \lesssim \sigma^2\norm{w_\sigma}^2_{H^2(\R)}\lesssim \sigma^2\norm{v_\sigma}^2_{H^2(\R)}
	 	\end{align*} 
	which proves \eqref{Est-I1R}.

\subsection{ Proof of \eqref{Est-Ij} when $j=2$}
By Plancherel and Cauchy-Schwarz inequality,
	 \begin{align*}
	 |I_2|=\bigg|\int_{\R}\partial_xv_\sigma.N_2(v_\sigma)dx\bigg| &= \bigg| \int_{\R}\langle D \rangle \partial_xv_\sigma. \langle D \rangle^{-1}N_2(v_\sigma)dx  \bigg|
	 \\
	 & \leq \norm{\langle D \rangle \partial_xv_\sigma}_{L^2_x(\R)} \norm{\langle D \rangle^{-1}N_2(v_\sigma)}_{L^2_x(\R)} \\& \lesssim \norm{v_\sigma}_{H^2(\R)}\norm{N_2(v_\sigma)}_{H^{-1}(\R)}.
	 \end{align*}

So the proof reduces to
	 	\begin{equation}
	 	 \label{Est-I2R}
	 	\norm{ N_2(v_\sigma)}_{H^{-1}(\R)}\lesssim \sigma^2 \norm{v_\sigma}^2_{H^2(\R)},
	 	\end{equation}
	 	where
	 	 	 	 $$
	 	 	 	 N_2(v_{\sigma})=(\partial_xv_\sigma)^2-\cosh(\sigma|D|) \left[ \sech(\sigma|D|) \partial_ xv_\sigma \right]^2.
	 	 	 	 $$
	 	Taking the spatial Fourier Transform of $N_2(v_\sigma)$ and using \eqref{coshpest} with $p=2$, we obtain
	 	\begin{align*}
	 	\Big|\mathcal F &\left[ N_2(v_\sigma)\right] (\xi)\Big|&
	 \\
	 &	=\Bigg|\int_{\xi=\xi_1+\xi_2}  
	 	\left( 1-\cosh(\sigma|\xi|)\prod_{j=1}^{2}\sech(\sigma|\xi_j|) \right)  i \xi_1\widehat{v_\sigma}(\xi_1)  i \xi_2 \widehat{v_\sigma}(\xi_2) d\xi_1d\xi_2 \Bigg| \\
	 	&\leq 8\sigma^2 \int_{\xi=\xi_1+\xi_2}   |\xi_1|^2 |\xi_2|^2 |\widehat{v_\sigma}(\xi_1)||\widehat{v_\sigma}(\xi_2)|d\xi_1d\xi_2
	 	\\
	 	&\leq 8\sigma^2\mathcal{F}_x\left[    |D|^2 w_\sigma.|D|^2 w_\sigma\right] (\xi),
	 	\end{align*} 
	 	where $w_\sigma=\mathcal{F}^{-1}_x(|\widehat{v_\sigma}|)$.
	 	
	 Then by  Plancherel, the Sobolev embedding
	 	$$H^{1}_x(\R)  \hookrightarrow L^{\infty}_x(\R)  \quad \Leftrightarrow  \quad L^{1}_x(\R) \hookrightarrow H^{-1}_x(\R)$$ and  Cauchy-Schwarz, we obtain
	 	\begin{align*}
	 	\norm{    N_2(v_\sigma) }_{H^{-1}_x(\R)}&\le  8\sigma^2 \norm{ |D|^2 w_\sigma.|D|^2w_\sigma  }_{H^{-1}_x(\R)}\\
	&\lesssim\sigma^2 \norm{  |D|^2 w_\sigma.|D|^2w_\sigma }_{L^1_x(\R)}
	\\& \lesssim \sigma^2\norm{ |D|^2 w_\sigma}_{L_x^2(\R)} \norm{ |D|^2w_\sigma}_{L_x^2(\R)}\\&\lesssim \sigma^2\norm{v_\sigma}^2_{H^2_x(\R)}.
	 	\end{align*} 
	 		which proves \eqref{Est-I2R}.

\subsection{ Proof of \eqref{Est-I3}}
By Cauchy-Schwarz inequality,
	 	\begin{align*}
	 	|I_3|=\frac{1}{8}\bigg|\int_{\R}\partial_xv_\sigma N_3(v_\sigma) dx\bigg| & \lesssim \norm{\partial_xv_\sigma}_{L^2_x(\R)}\norm{N_3(v_\sigma)}_{L^2_x(\R)}\\& \lesssim \norm{v_\sigma}_{H^1(\R)}\norm{N_3(v_\sigma)}_{L^2_x(\R)}.
	 	\end{align*}
	 	So it remains to prove
	 	\begin{equation}
	 	 \label{Est-I3R}
	 	\norm{ N_3(v_\sigma)}_{L^2_x(\R)}\lesssim \sigma^2 \norm{v_\sigma}^3_{H^2(\R)},
	 	\end{equation}
	 	where
	 	 	 	 $$
	 	 	 	N_3(v_{\sigma})=v_{\sigma}^3-\cosh(\sigma|D|) \left[ \sech(\sigma|D|) v_\sigma \right]^3.
	 	 	 	 $$
	 	Taking the Fourier Transform of $N_3(v_\sigma)$ and applying \eqref{coshpest} with $p=3$, we obtain
	 	\begin{align*}
	 	\Big|\mathcal F_x &\left[  N_3(v_\sigma)\right] (\xi)\Big|&
	 \\
	 &	=\Bigg|\int_{\xi=\xi_1+\xi_2 + \xi_3}  
	 	\left( 1-\cosh(\sigma|\xi|)\prod_{j=1}^{3}\sech(\sigma|\xi_j|) \right) \widehat{v_\sigma}(\xi_1) \widehat{v_\sigma}(\xi_2) \widehat{v_\sigma}(\xi_3)  d\xi_1d\xi_2 d\xi_3\Bigg| \\
	 	&\leq 8\sigma^2 \int_{\xi=\xi_1+\xi_2+ \xi_3}  
	 	\left( \sum_{  j\neq k =1}^3   |\xi_j| |\xi_k| \right)|\widehat{v_\sigma}(\xi_1)||\widehat{v_\sigma}(\xi_2)| |\widehat{v_\sigma}(\xi_3)|d\xi_1d\xi_2 d\xi_3
	 	\end{align*} 
	 	By symmetry, we may assume $|\xi_1|\leq |\xi_2|\leq |\xi_3|$, which implies 
	 	\begin{align*}
	 	\Big|\mathcal F_x \left[  N_3(v_\sigma)\right] (\xi)\Big| &\leq 48\sigma^2\int_{\xi=\xi_1+\xi_2+\xi_3} |\widehat{v_\sigma}(\xi_1)||\widehat{v_\sigma}(\xi_2)||\xi_3|^2|\widehat{v_\sigma}(\xi_3)| d\xi_1d\xi_2d\xi_3
	 	\\
	 	&=48 \sigma^2\mathcal{F}_x(w_\sigma.w_\sigma.|D|^2 w_\sigma)(\xi),
	 	\end{align*}
	 	where $w_\sigma=\mathcal{F}^{-1}_x(|\widehat{v_\sigma}|)$.
	 	
	Then by Plancherel and H\"older inequality we get
	 \begin{align*}
	 \norm{\mathcal F_x \left[  N_3(v_\sigma)\right] (\xi)}_{L^2_x(\R)}\lesssim \sigma^2\norm{w_\sigma^2.|D|^2w_\sigma}_{L^2_x(\R)} &\lesssim \sigma^2 \norm{w_\sigma}^2_{L^{\infty}_x(\R)}\norm{|D|^2w_\sigma}_{L^2_x(\R)}\\& \lesssim \sigma^2 \norm{w_\sigma}^2_{H^2(\R)} \norm{w_\sigma}_{H^2(\R)}\\&\lesssim \sigma^2 \norm{w_\sigma}^3_{H^2(\R)}\\& \lesssim\sigma^2 \norm{v_\sigma}^3_{H^2(\R)}
	 \end{align*} 
 	 		which proves \eqref{Est-I3R}.

\vspace{10mm}

\noindent \textbf{Acknowledgments}
A. Tesfahun acknowledges support from the Social Policy Research Grant (SPG), Nazarbayev University.

\end{document}